\documentclass[12pt,reqno]{amsart}
\usepackage{epsfig,amsmath,amsfonts,latexsym,accents}
\usepackage[usenames,dvipsnames]{color}
\usepackage{palatino}
\usepackage[hypcap=false]{caption}
\usepackage{caption, subcaption, pinlabel}
\usepackage{color} \usepackage{marginnote}
\usepackage[top=1.5cm, bottom=1.5cm, outer=2cm, inner=2cm, heightrounded, marginparwidth=3cm, marginparsep=.2cm]{geometry}

\usepackage{xcolor}
\usepackage{enumitem}
\definecolor{darkgreen}{rgb}{0, 0.5, 0}

\usepackage[hyperfootnotes=false]{hyperref}
\hypersetup{
  colorlinks,
  citecolor=Purple,
  linkcolor=Black,
  urlcolor=Green}

\theoremstyle{plain}
\newtheorem{dummy}{anything}[section]
\newtheorem{theorem}[dummy]{Theorem}

\newtheorem{lemma}[dummy]{Lemma}

\newtheorem{question}[dummy]{Question}

\newtheorem{proposition}[dummy]{Proposition}

\newtheorem{corollary}[dummy]{Corollary}

\theoremstyle{definition}

\newtheorem{remark}[dummy]{Remark}

\theoremstyle{remark}

\newcommand{\C}{\mathbb{C}}

\def\rr{\bold r}

\def\b{\beta}

\def\s{\sigma}

\begin{document}

\title {A characterization of quasipositive two-bridge knots}

\author{\hspace{3cm} Burak Ozbagci \newline (\MakeLowercase{with an Appendix by} Stepan Orevkov)}

\address{Institut de Math\'{e}matiques de Toulouse, Universit\'{e} Paul Sabatier, 118 route de Narbonne, 31062 Toulouse, France}

\address{Steklov Mathematical Institute, Gubkina 8, Moscow, Russia}

\email{stepan.orevkov@math.univ-toulouse.fr}

\address{Department of Mathematics, Ko\c{c} University, 34450, Istanbul,Turkey}
\email{bozbagci@ku.edu.tr}

\subjclass[2000]{}
\thanks{}


\begin{abstract} We prove a simple necessary and sufficient condition for a two-bridge knot $K(p,q)$ to be quasipositive, based on the continued fraction expansion of $p/q$. As an application, coupled with some classification results in contact and symplectic topology, we give a new proof of the fact that smoothly slice two-bridge knots are non-quasipositive.  Another proof of this fact using methods within the scope of knot theory is presented in the Appendix.
 \end{abstract}

\maketitle

\section{Introduction}

Notions of quasipositivity and strong quasipositivity for links were introduced and explored by Rudolph in a series of papers (see, for example, \cite{r1,r, r3, r4, r2, ru}).
Let $\s_1, \ldots, \s_{n-1}$ denote the standard generators  of the braid group $B_n$,  and let $\s_{i,j} = (\s_i \ldots \s_{j-2})(\s_{j-1}) (\s_i \ldots \s_{j-2})^{-1}.$ Strongly quasipositive links are the links which can be realized as the closure of braids of the form $\prod_{k=1}^{m} \s_{i_k, j_k}$. The weaker notion of  quasipositive link is any link which can be realized as the closure of a braid of the form $\prod_{k=1}^{m} w_k \s_{i_k} w_k^{-1}$, where $w_k \in B_n$ for all  $1 \leq k \leq m$.  An  oriented link is called positive if it has a positive diagram, i.e., a diagram in which all crossings are positive.

Throughout this paper, we assume that $p > q$ are relatively prime positive integers. The oriented lens space $L(p,q)$ is defined by $-p/q$ surgery on the unknot in $S^3$ and it is well-known that $L(p,q)$ is the double cover of $S^3$ branched along  the two-bridge link  $K(p,q)$, which we depicted in Figure~\ref{fig: 2bridge}.  When $p$ is odd, $K(p,q)$ is a knot, and otherwise it has two components.

The
{\em negative} continued fraction of $p/q$ is defined  by
$$ \frac {p}{q} = [a_1,a_2,\ldots,a_k]^{-} = a_1-
\cfrac{1}{a_2- \cfrac{1}{\ddots- \cfrac{1}{a_k}}} \;, \qquad a_i\geq 2
\text{ for all $1\leq i\leq k,$} $$ where the coefficients $a_1,a_2,\ldots,a_k$ are uniquely determined by $p/q$.  We say that the negative continued fraction of $p/q$ is even if $a_i$ is even for all $1\leq i\leq k$.

\begin{theorem}\label{thm: main}
 If the two-bridge link $K(p,q)$ is  quasipositive then the negative continued fraction of $p/q$ is even. \end{theorem}

If  the negative continued fraction of $p/q$ is even, then $pq$ must be even. Corollary~\ref{cor: odd} immediately follows from this simple observation coupled with Theorem~\ref{thm: main}.

\begin{corollary} \label{cor: odd} If $pq$ is odd, then the two-bridge knot $K(p,q)$ is non-quasipositive.
\end{corollary}

Conversely, we can express Tanaka's quasipositivity criterion \cite[Proposition 3.1]{t} for two-bridge knots in terms of negative continued fractions as follows:

\begin{proposition}[Tanaka \cite{t}] \label{prop: tan} The two-bridge knot $K(p,q)$ is  strongly quasipositive, provided that  the negative continued fraction of $p/q$ is even. \end{proposition}

We would like to point out  that Proposition~\ref{prop: tan} cannot possibly hold for arbitrary two-bridge links.  For example, $16/3 =[6,2,2]^{-}$ and $16/13=[2,2,2,2,4]^{-}$,   but the two-bridge link $K(16,3)$ and its mirror  $K(16,13)$ cannot both be quasipositive since Hayden \cite[Corollary 1.5]{h} proved that if a link and its mirror are both quasipositive, then the link is an
unlink.

Combining Theorem~\ref{thm: main} and Proposition~\ref{prop: tan}, we have the following characterization  of  quasipositive two-bridge knots.

\begin{corollary} \label{cor: even} The two-bridge knot $K(p,q)$ is quasipositive if and only if the negative continued fraction of $p/q$ is even.
\end{corollary}

\begin{remark} \label{rem: equiv} It is clear that strong quasipositivity implies quasipositivity by definition. Conversely, Boileau and Rudolph \cite[Proposition 3.6 and Corollary 3.7]{br}  proved for a family of alternating arborescent links, including 2-bridge links, that quasipositivity implies strong quasipositivity. On the other hand, by the work of Rudolph \cite{r2} and Nakamura \cite[Lemma 4.1]{n}, positive diagrams represent strongly quasipositive links. Moreover, strongly quasipositive two-bridge links are positive, since any two-bridge link can be obtained as the boundary of plumbings of annuli and strong quasipositivity behaves naturally  with respect to the plumbing operation \cite{r4}. The upshot is that, positivity, quasipositivity and strong quasipositivity are in fact equivalent  for two-bridge links. \end{remark}

The proof of Theorem~\ref{thm: main}  is based on the work of Plamenevskaya \cite{p}, coupled with the following result on contact topology.

 \begin{proposition}\label{prop: chern}
The lens space $L(p,q)$ admits a tight contact structure with trivial first Chern class  if and only if the negative continued fraction of $p/q$ is even. \end{proposition}

Recall that a knot in $S^3$ is called {\em smoothly slice} if it bounds a smooth properly embedded disk in $B^4$. Rudolph \cite[Proposition 2]{r3} showed that the only smoothly slice strongly quasipositive knot is the unknot, as a corollary to the celebrated work of Kronheimer and Mrowka \cite{km}. As we pointed out in Remark~\ref{rem: equiv}, quasipositivity and strong quasipositivity are equivalent for a two-bridge link.  Therefore, we conclude that smoothly slice two-bridge knots are non-quasipositive. We will provide a new proof of this fact as an application of Corollary~\ref{cor: even} coupled with the following result on symplectic topology.

\begin{proposition} \label{prop: stein}   Let $(Y, \xi)$  be the contact double cover of the standard contact $3$-sphere $(S^3, \xi_{st})$ branched along the knot $K$ which we assume to be transverse to $\xi_{st}$. If $K$ is smoothly slice and quasipositive, then $(Y, \xi)$ admits a rational homology ball Stein filling.  \end{proposition}

After we finished a first draft of this paper, it came to our attention that in \cite[Remark 2]{o},   Orevkov observed that quasipositivity
  implies (by results of \cite{o} combined with \cite{deh})
  strong  quasipositivity for two-bridge links,
  and for a larger class of links including all
  alternating Montesinos links---which does not directly follow from the arguments in \cite[Proposition 3.6 and Corollary 3.7]{br}. Since we convinced him that his result would be of interest to knot enthusiasts, he kindly agreed to write the details of his  \cite[Remark 2]{o} in the Appendix (especially in the two-bridge case where
   \cite{mu1} can be used instead of \cite{deh}), which in turn,  gives yet another proof of the fact that smoothly slice two-bridge knots are non-quasipositive.

\section{Applications of  contact and symplectic topology}

We begin with the proof of Proposition~\ref{prop: chern}.

\begin{proof}[Proof of Proposition~\ref{prop: chern}]  Suppose that $p/q=[a_1,a_2,\ldots,a_k]^{-}$, where $a_i \geq 2$ is even for all $1\leq i\leq k$. Note that $L(p,q)$ can be obtained by surgery on a chain of unknots with framings $-a_1,-a_2,\ldots,-a_k$, respectively. Since  $a_i \geq 2$ is even for all $1\leq i\leq k$, we can Legendrian realize each unknot  in the chain, with respect to the standard contact structure,   so that the rotation number of each Legendrian unknot is zero. It follows by \cite[Proposition 2.3]{g} that the first Chern class of the Stein fillable (and hence tight) contact structure on $L(p,q)$ obtained by Legendrian surgery on the resulting Legendrian link is zero.

To prove the only if direction, suppose that $p/q=[a_1,a_2,\ldots,a_k]^{-}$, where $a_i \geq 2$ for all $1\leq i\leq k$ and $a_j$ is odd for some $1\leq j \leq k$. Let $\xi$ be the Stein fillable contact structure on $L(p,q)$ obtained by Legendrian surgery along an arbitrary Legendrian realization $\mathcal{L}$ of the chain of unknots with smooth framings $-a_1,-a_2,\ldots,-a_k$, respectively. Let $\overline{\mathcal{L}}$ be the Legendrian link obtained from $\mathcal{L}$ by taking the mirror image of each component and   let $\overline{\xi}$ be the Stein fillable contact structure on $L(p,q)$ obtained by Legendrian surgery along $\overline{\mathcal{L}}$. It follows that rotation number of each component of $\overline{\mathcal{L}}$ is the negative of the rotation number of the corresponding component of $\mathcal{L}$ and hence $\overline{\xi}$ is obtained from $\xi$ by reversing the orientation of the contact planes. By Honda's classification \cite[Theorem 2.1]{ho} of tight contact structures on $L(p,q)$, the contact structures  $\xi$ and $\overline{\xi}$ are not isotopic because of the assumption that $a_j$ is odd for some $1\leq j \leq k$. Note that non-isotopic tight contact structures on $L(p,q)$ are non-homotopic \cite[Proposition 4.24]{ho}. This implies that $c_1(\xi) \neq 0$ since according to Gompf \cite[Corollary 4.10]{g}, an oriented plane field in any closed oriented $3$-manifold is homotopic to itself with reversed orientation if and only if it has trivial first Chern class.
\end{proof}

We now give a proof of Theorem~\ref{thm: main},  based on Proposition~\ref{prop: chern}.

\begin{proof}[Proof of Theorem~\ref{thm: main}] The double cover of $S^3$ branched along $K(p,q)$ is $L(p,q)$. We can make $K(p,q)$ transverse to the standard contact structure $\xi_{st}$ by isotopy. Let $(L(p,q),  \xi)$ be the contact double cover of $(S^3, \xi_{st})$ branched along the transverse link $K(p,q)$. Suppose that $K(p,q)$ is quasipositive. By the work of Plamenevskaya \cite[Proposition 1.4 and Lemma 5.1]{p},  we conclude that $\xi$ is Stein fillable and moreover $c_1(\xi)=0$. It follows by Proposition~\ref{prop: chern} that
the negative continued fraction of $p/q$ must be  even.
\end{proof}

The proof of Proposition~\ref{prop: tan} is essentially obtained by rephrasing Proposition 3.1 in Tanaka's paper \cite{t}, where he uses regular continued fractions to describe a sufficient condition for a two-bridge knot to be strongly  quasipositive.

\begin{proof}[Proof of Proposition~\ref{prop: tan}]
A  {\em regular} continued fraction of $p/q$ is defined by $$ \frac {p}{q} = [c_1,c_2,\ldots,c_{2m+1}]= c_1+ \cfrac{1}{c_2+ \cfrac{1}{\ddots+ \cfrac{1}{c_{2m+1}}}} \;, \qquad c_i > 0
\text{ for all  $1\leq i\leq 2m+1$}. $$

Note that there is always a regular continued fraction of $p/q$ of odd length, as above. To see this, suppose that $p/q$  has a regular continued fraction of even length, i.e.  $ \frac {p}{q} = [c_1,c_2,\ldots,c_{2m}]$ with $c_i >0$. If $c_{2m}=1$, then $$ \frac {p}{q} = [c_1,c_2,\ldots,c_{2m}] =[c_1,c_2,\ldots, c_{2m-2}, 1+c_{2m-1}]$$ and  otherwise, $$ \frac {p}{q} =[c_1,c_2,\ldots,c_{2m}] =[c_1,c_2,\ldots, c_{2m-1}, -1+ c_{2m}, 1].$$

Using an odd length regular continued fraction $\frac {p}{q} = [c_1,c_2,\ldots,c_{2m+1}]$, we define the two bridge link $K(p,q)$ as depicted in Figure~\ref{fig: 2bridge}, where the integer inside each box denotes the signed number of  half twists to be inserted.
 \begin{figure}[h]
\centering
 {\epsfxsize=6in
\centerline{\epsfbox{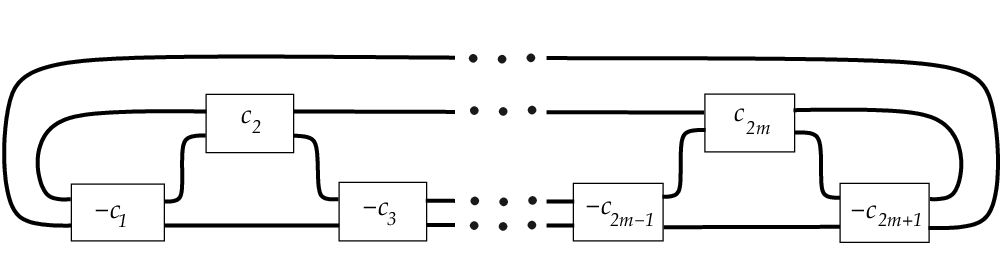}}}
 \caption{Two-bridge link $K(p,q)$, where $ p/q = [c_1,c_2,\ldots,c_{2m+1}].$}
                        \label{fig: 2bridge}
 \end{figure}

In \cite{t},  Tanaka uses the notation $\mathcal{C} (c_1, c_2, \ldots, c_{2m+1})$, with each  $c_i>0$, to describe a two-bridge knot. By comparing \cite[Figure 6]{t}, for example,  with our Figure~\ref{fig: 2bridge}, we see that our definition of $K(p,q)$ is the {\em mirror image} of the one described by Tanaka, but agrees with the one  described by Lisca \cite{l}.

According to  \cite[Proposition 3.1]{t},  the two-bridge knot $\mathcal{C} (c_1, c_2, \ldots, c_{2m+1}) $, which is the mirror image  of our $K(p,q)$ defined above,  is strongly quasipositive as long as $c_i$ is even for each even index $2 \leq i \leq 2m$.  To express this condition in terms of negative continued fractions,   we just describe how to convert a given regular continued fraction of any $p/q$ to the negative continued fraction of $p/(p-q)$ and vice versa in Lemma~\ref{lem: convert}. This finishes the proof of Theorem~\ref{thm: main}, since by Lemma~\ref{lem: convert}, we can easily deduce that the negative continued fraction  of $p/(p-q)$ is even if and only if $p/q$ has a regular continued fraction of odd length where each even indexed coefficient is even.
\end{proof}

\begin{lemma} \label{lem: convert} Suppose that $p/q = [c_1,c_2,\ldots,c_{2m+1}]$ with $c_i >0$ for all   $1\leq i\leq 2m+1$.  Then $$  \frac{p}{q} =
[1+c_1, \underbrace{2, \ldots, 2}_{c_2-1}, 2+c_3, \underbrace{2, \ldots, 2}_{c_4-1},  \ldots, 2+c_{2m-1}, \underbrace{2, \ldots, 2}_{c_{2m}-1},  1+c_{2m+1}]^{-}$$  and
$$ \frac{p}{p-q} =
[\underbrace{2, \ldots, 2}_{c_1-1}, 2+c_2, \underbrace{2, \ldots, 2}_{c_3-1}, 2+c_4, \ldots, \underbrace{2, \ldots, 2}_{c_{2m-1}-1}, 2+c_{2m}, \underbrace{2, \ldots, 2}_{c_{2m+1}-1}]^{-}$$

\end{lemma}

\begin{proof} The negative continued fraction of $p/q$ can be obtained from a given regular continued fraction of $p/q$ by a straightforward induction argument, whereas the negative continued fraction of $p/(p-q)$ is obtained from that of $p/q$ by the Riemenschneider’s point diagram method \cite{rie}.
\end{proof}

Using Lemma~\ref{lem: convert}, we can rephrase Corollary~\ref{cor: even} as follows: 

\begin{corollary} \label{cor: ev} The two-bridge knot $K(p,q)$ is quasipositive if and only if $p/(p-q)$ has a regular continued fraction of odd length where each even indexed coefficient is even.
\end{corollary}

Finally, we turn our attention to  Proposition~\ref{prop: stein}.

\begin{proof}[Proof of Proposition~\ref{prop: stein}]   Rudolph \cite{r1} showed that quasipositive links  arise as the transverse intersection of $S^3 \subset  \C^2$, with a complex curve. Therefore, the quasipositivity assumption implies that $K$ bounds  a complex curve  in $B^4$. Since  $K$ is assumed to be smoothly slice as well,  there is a smooth disk in $B^4$ with boundary $K$. But by the "local Thom conjecture" \cite[Corollary 1.3]{km} , the complex curve minimizes genus, so the slice disk can be assumed to be complex. Hence, the analytic double cover of  $B^4$ equipped with its standard complex structure, branched along this complex slice disk, is a rational homology ball Stein filling of $(Y, \xi)$. It is well-known that the double cover is a rational homology ball and it is in fact Stein by \cite[Theorem 3]{lp}.\end{proof}

\section{Smoothly slice two-bridge knots are non-quasipositive}

As we pointed out in the Introduction, the fact that smoothly slice two-bridge knots are non-quasipositive follows by combining \cite[Proposition 2]{r3} with \cite[Proposition 3.6 and Corollary 3.7]{br}.   Here we provide an alternate proof based on contact and symplectic topology.

\begin{corollary}   \label{cor: slice} Smoothly slice two-bridge knots are non-quasipositive.
\end{corollary}

\begin{proof}   Lisca  \cite{l}  showed that $L(p,q)$ bounds a rational homology ball  if and only if $p/q$ belongs to a certain subset $\mathcal{R}$ of the set of positive rational numbers. Lisca also showed that, for odd $p$,  any two-bridge knot $K(p,q)$ is smoothly slice if and only if $p/q \in  \mathcal{R}$. By definition, the set $\mathcal{R}$ contains a subset, denoted by $\mathcal{O}$ here, which consists of $p/q >0$ such that $p=m^2$ (for odd $m$), and $q=mh-1$ where $0< h < m$ and $(m,h)=1$.

Suppose that $p/q \in \mathcal{R} \setminus \mathcal{O}$ and  $K(p,q)$  is a smoothly slice, quasipositive two-bridge knot.   We can isotope $K(p,q)$ to be transverse to the standard contact structure $\xi_{st}$  in $S^3$. Let $(L(p,q), \xi)$ be the double cover of $(S^3, \xi_{st})$ branched along the transverse knot $K$. According to Proposition~\ref{prop: stein},  $(L(p,q), \xi)$ admits a rational homology ball Stein filling, which in turn, implies that $\xi$ is isomorphic to the canonical contact structure  $\xi_{can}$, because no virtually overtwisted lens space has a rational homology ball symplectic filling by the work of Golla and Starkston \cite{gs}. (See also \cite[Lemma 1.5]{er}, \cite[Proposition 11]{et}). This gives a contradiction to the fact that $(L(p,q), \xi_{can})$ admits a rational homology ball symplectic filling  if and only if $p/q \in \mathcal{O}$, as shown by Lisca \cite[Corollary 1.2(c)]{l}.

Now suppose that $p/q \in \mathcal{O}$,  i.e.
$p=m^2$, for odd $m>1$ and $q=mh-1$ where $0< h < m$ and $(m,h)=1$. If $h$ is even, then $q=mh-1$ is odd and $K(p,q)$ is  non-quasipositive by Corollary~\ref{cor: even}. On the other hand, if $h$ is odd, then $q=mh-1$ is even but $q' = m (m-h)-1$ is odd and $qq' \equiv 1(\mbox{mod}\; m^2).$ Since $K(p,q)$ is isotopic to $K(p, q')$ \cite[Chapter 12]{bz}, we conclude that  $K(p, q')$ and hence $K(p, q)$ is non-quasipositive, again by Corollary~\ref{cor: odd}.
\end{proof}

\noindent {\bf {Acknowledgment}}: We would like to thank Sebastian Baader,  Peter Feller, Marco Golla, and Andr\'{a}s Stipsicz for their useful comments on a draft of this note.

\section {Appendix by Stepan Orevkov}

The Seifert graph of a connected link diagram $D$ is the graph $G_D$ whose vertices correspond to Seifert circles and the edges correspond to the crossings. Each edge is equipped with the sign of the corresponding crossing. We say that a diagram $D$ is reduced if $G_D$  does not have any edge whose removal disconnects $G_D$. Let $w(D)$ denote the writhe of $D$, which is  the sum of the signs of all crossings of $D$.

Suppose that $D$ is an alternating diagram of a link $L$. Let $b = b(L)$ denote the braid index of $L$ and $s = s(D)$ denote the number of Seifert circles of $D$. We define $d^{\pm}=d^{\pm}(D)$ as the number of edges of $\pm$ sign in a spanning tree of $G_D$.   It is easy to check that these quantities do not depend on the choice of $G_D$ when $D$ is alternating (this fact can be also derived from Traczyk's theorem [25, Theorem 3]).

Suppose that $\b$ is a braid with $b$ strands realizing $L$. Due to Dynnikov–Prasolov Theorem \cite{dp}, $w(\b)$ does not depend on the choice of $b$-strand $\b$ realizing $L$, which allows one to define the numbers $\rr^{\pm} = \rr^{\pm} (D)$ from the system of equations
$$\rr^{+} + \rr^- =s-b \quad \mbox{ and } \quad \rr^{+} - \rr^{-} = w(D)-w(\b).$$

By the work of Rudolph \cite{r2} and Nakamura \cite[Lemma 4.1]{n}, positive diagrams represent strongly quasipositive links. Conversely,  Baader \cite{b} showed that homogeneous  strongly quasipositive knots are positive.  Note that the class of homogeneous links, introduced by Cromwell \cite{c},  includes all
alternating links. Moreover,  alternating strongly quasipositive
links have positive alternating diagrams by Boileau, Boyer and Gordon \cite[Corollary 7.3]{bbg}.  Therefore, the following question appears naturally (see \cite{b}, for example):

\begin{question}\label{que: baader} Is it true that alternating quasipositive links have positive diagrams?
\end{question}

A partial answer to this question was provided as follows:

\begin{theorem}\label{thmo} {\rm(\cite[Theorem 4]{o}).}
If $D$ is a reduced alternating diagram of a quasipositive link $L$, which satisfies the inequality
\begin{equation} \label{eq: a}
    2\rr^-(D) \le d^-(D)
\end{equation} then $D$ is positive, and hence $L$ is strongly quasipositive. \end{theorem}

\begin{proposition}\label{propo} {\rm(\cite[Remark 2]{o}).}
The inequality~(\ref{eq: a}) is satisfied by a standard alternating diagram of any two-bridge link or any alternating Montesinos link.
\end{proposition}

\begin{remark} \label{rem: dhl}As pointed out in \cite[Remark 2]{o},
inequality~(\ref{eq: a}) is actually proven in \cite{deh} for the diagrams
from \cite[Thm. 4.10, 4.12, 4.14]{deh} (in particular, for those in
Proposition~\ref{propo})
even though it is not formulated in
\cite{deh} explicitly. Since it is not so easy to recognize the proof of this fact without carefully
reading the whole paper \cite{deh}, one of our goals here is to
help the reader to extract this proof from \cite{deh}.
\end{remark}

\begin{remark} \label{rem: orev} Note that by Proposition~\ref{propo}, if a two-bridge link or an alternating Montesinos link is quasipositive, then
it is positive. \end{remark}

The statement in Proposition~\ref{prop: ineq} was claimed in \cite[Remark~2]{o} and as mentioned there,  it follows from \cite{deh}. For rational links, however, it can also be easily derived from Murasugi's work \cite[Section 3]{mu1}.

\begin{proposition}  \label{prop: ineq} Every oriented two-bridge (aka rational)  link
admits an alternating diagram  $D$ satisfying the inequalities $2\rr^-(D) \le d^-(D)$ and  $2\rr^+(D) \le d^+(D)$.
 \end{proposition}

\begin{proof}  For rational links we follow the orientation convention used in Murasugi's book \cite{mu}.
Let $L$ be a rational oriented link of type $(p,q)$. For the mirror image $D^*$ of $D$, we have $d^{\pm}(D^*) = d^{\mp}(D)$ and $\rr^{\pm}(D^*) = \rr^{\mp}(D)$.
Therefore, without loss of generality we may assume that $q$ is odd and $0<q<p$.
Using the notation introduced in \cite[Section 3]{mu1}, let
$$
   \frac{p}{p-q}=[2n_{1,1},\ldots,2n_{1,k_1},\;
   -2n_{2,1},\ldots,-2n_{2,n_2},\; \ldots\;,
         (-1)^{t-1}2n_{t,1},\ldots,(-1)^{t-1}2n_{t,k_t}]^-,
     $$
where we assume $n_{i,j}>0$.

Let $b$ be the braid index of $L$ and $e=w(\beta)$ be the exponent sum of
a $b$-braid $\beta$ representing $L$. By \cite[Prop.~4.2 and Thm.~4.3]{mu1}
we have
\begin{equation}\label{2}
   b = t+1+\sum_{i=1}^t\sum_{j=1}^{k_i}(n_{i,j}-1),
   \qquad
   e = \frac{1-(-1)^t}2 + \sum_{i=1}^t(-1)^{i-1}\sum_{j=1}^{k_i}n_{i,j}. 
\end{equation}
Using the standard properties of rational links and continued fractions, one
easily checks that $L$ admits an alternating diagram $D$ shown in Figure~2
where $T_i$ are the tangles defined by the braids
$$
   T_i=\begin{cases}
       \sigma_1^{1-2n_{i,1}}\Big(\prod_{j=2}^{k_i}\sigma_2\sigma_1^{2-2n_{i,j}}\Big)
       \sigma_1^{-1}, &\text{$i$ is odd,}
       \\
       \sigma_2^{2n_{i,1}-1}\Big(\prod_{j=2}^{k_i}\sigma_1^{-1}\sigma_2^{2n_{i,j}-2}\Big)
       \sigma_2, &\text{$i$ is even,}
       \end{cases}
$$
as depicted in Figure~3 for odd $i$. Note that, for odd (resp. even) $i$,
all crossings of $T_i$ are positive (resp. negative).

\begin{figure}[h]
\centering
 {\epsfxsize=4.7in
\centerline{\epsfbox{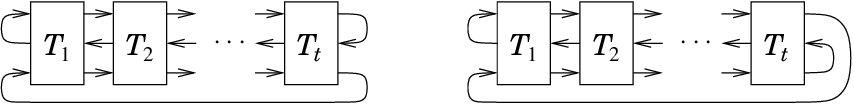}}}
  \centerline{$t$ is odd \hskip 50mm $t$ is even}
\caption{Alternating diagram $D$ for the rational link $L$.}
                        \label{fig: rem2}
 \end{figure}

\begin{figure}
\centerline{\epsfxsize=4.7in \epsfbox{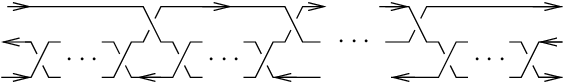}}
\centerline{\hskip 5mm $2n_{i,1}-1$ \hskip 14mm $2n_{i,2}-2$ \hskip 39mm $2n_{i,k_i}-1$}
 \caption{The tangle $T_i$ for an odd $i$.}
\end{figure}

Let $w=w(D)$ and $s=s(D)$ be the writhe and the number of Seifert circles of $D$.
Each tangle $T_i$ contributes $(-1)^{i-1}w_i$ to $w$, where
$w_i=1+\sum_j(2n_{i,j}-1)$. If $i$ is odd (resp. even), then $T_i$
contributes $w_i-k_i$ to $d^+$ (resp. to $d^-$). We have $s-(d^+ + d^-)=1$
(the Euler characteristic of a spanning tree).

Hence
\begin{equation}\label{3}
   s = t + 1 + \sum_{i=1}^t\sum_{j=1}^{k_i}2(n_{i,j}-1),
   \quad
   w = \frac{1-(-1)^t}2 + \sum_{i=1}^t(-1)^{i-1}\sum_{j=1}^{k_i}(2n_{i,j}-1), 
\end{equation}
\begin{equation}\label{4}
  d^+ = \lceil t/2 \rceil+\sum_{\text{$i$ odd}}\; \sum_{j=1}^{k_i}2(n_{i,j}-1), \qquad
  d^- =\lfloor t/2\rfloor+\sum_{\text{$i$ even}}\;\sum_{j=1}^{k_i}2(n_{i,j}-1). 
\end{equation}
By combining (\ref{2}) and (\ref{3}) we obtain
\begin{equation}\label{5}
   s-b = \sum_{i=1}^t\sum_{j=1}^{k_i}(n_{i,j}-1), \qquad
   w-e = \sum_{i=1}^t(-1)^{i-1}\sum_{j=1}^{k_i}(n_{i,j}-1).                   
\end{equation}
Recall that $\rr^\pm=\rr^\pm(D)$ are defined by
\begin{equation}\label{6}
   \rr^+ + \rr^- = s-b, \qquad \rr^+ - \rr^- = w-e.                         
\end{equation}
By combining (\ref{4}), (\ref{5}), and (\ref{6}) we obtain
$$
   2\rr^+ = d^+ - \lceil  t/2\rceil  \le d^+, \qquad
   2\rr^- = d^- - \lfloor t/2\rfloor \le d^-.
$$\end{proof}

\subsection{How to extract the proof of Proposition~\ref{propo} from \cite{deh}}\label{sec: dehl}

The paper \cite{deh} is devoted to a computation of the braid index of alternating links
presented by link diagrams of some specific forms.
An upper bound for the braid index is the number of Seifert circles.
A lower bound is given by Morton-Franks-Williams (MFW)
inequality (\cite{fw, m}). In many cases (those indicated in Theorem~\ref{thmo})
it is shown in \cite{deh} that the MFW bound is sharp. In each case, this is
done in \cite{deh} as follows. Given a reduced alternating diagram $D$ of a link $L$,
one chooses a certain collection $C$
of lone crossings, and applies the Murasugi-Przytycki move \cite{mp}
(MP-move) to each of them
(MP-moves are also depicted in \cite[Fig.~2]{t} and \cite[Fig.~11]{dhl}).
The number of Seifert circles of the resulting diagram $D'$
(in general, non-alternating) is equal to the braid index of $L$ and, moreover,
the number of the performed MP-moves at positive (negative) crossings is equal to
$\rr^+(D)$ (resp. $\rr^-(D)$).

A {\it constant-sign path of length} $n$ in $G_D$ is a sequence
$v_1,\dots,v_n$ of pairwise distinct vertices of $G_D$ such that each pair of
consecutive vertices $(v_i,v_{i+1})$ is connected by an edge, and all these
edges are of the same sign.
One can check that the vertices of $G_D$ corresponding to $C$
are always chosen in \cite{deh} in some pairwise disjoint constant-sign paths.
Moreover, at most $\lfloor(n-1)/2\rfloor$ crossings is chosen in each of
the paths of length $n$.
It is clear that any collection of pairwise disjoint paths is contained in some spanning tree. Thus we obtain (\ref{eq: a}) for all diagrams mentioned in Remark~\ref{rem: dhl},
in particular, this gives a proof of Proposition~\ref{propo}.

\end{document}